\documentclass[12pt]{article}

\usepackage{amssymb}
\usepackage{latexsym}
\usepackage{graphicx}
\usepackage{amsmath}
\usepackage{amsthm}
\usepackage{amsfonts}
\usepackage{amscd}

\usepackage{comment}

\usepackage{indentfirst}

\theoremstyle{plain}
\newtheorem{theorem}{Theorem}[section]
\newtheorem{proposition}{Proposition}[section]
\newtheorem{lemma}{Lemma}[section]
\newtheorem{definition}{Definition}[section]
\newtheorem{corollary}{Corollary}[section]

\newtheorem{example}{Example}[section]
\newtheorem{remark}{Remark}[section]

\title{\    Hamiltonian cycles for the square of the augmentation graphs and  Gray codes for
restricted permutations and ascent sequences}
\author{Masaya Tomie}
\date{Morioka University, Takizawa-shi, Iwate 020-0694, Japan 
(e-mail: tomie@morioka-u.ac.jp)}

\begin{document}

\maketitle

\begin{abstract}

In this paper, we construct  a listing for the vertices of the augmentation 
graph of given size, and as a consequence, we obtain a Hamiltonian 
cycle for the square of the augmentation graph of given size. 
As applications, we have a Gray code for the 
$132$-$312$ avoiding permutations of given length  
such that two successive permutations differ by at most $2$
 adjacent transpositions. 
Also we obtain   Gray codes of strong distance $2$   
 for the $001$ avoiding ascent sequences and the $010$ avoiding ascent sequences  
 of  given length.

\end{abstract}

\section{Introduction}\label{introduction}

The  \textit{combinatorial Gray codes} is a mathematical model 
to generate the given combinatorial object such that each element is generated exactly once and 
any successive elements differ  in some pre-specified, usually small, way. 
There are many examples of 
 minimal change listings of combinatorial objects.
Though there are many bijections among the combinatorial objects 
of same cardinality, 
 in general, a bijection does not preserve the Gray code structures 
and hence we should construct Gray codes by ad hoc methods in many 
cases \cite{savage}.

It is known that 
a Gray code problem can be formulated as a Hamiltonian path or  a Hamiltonian cycle problem: 
the vertices of the graph are the objects and two vertices are joined by an edge
 if they differ from each other in the pre-specified way.

The {\it augmentation graph} of size $n$, 
see Section  {\rmfamily \ref{hamiltonian-augmentation}}, 
 can be defined on the binary words of length $n$.
Douglas West posed the problem to determine whether there is a Hamiltonian path 
in the augmentation graph. 
When $\frac{n(n-1)}{2}$ is even, there are no Hamiltonian paths for the 
augmentation graph of size $n$ and 
to the best of our knowledge, the problem is open for $n \ge 7$ with $\frac{n(n-1)}{2}$ being odd \cite{savage}. 
In  Section {\rmfamily \ref{hamiltonian-augmentation}}, we construct a Hamiltonian cycle for the square of the augmentation graph of size $n$.

The theory of pattern avoidance is a very active research area and many papers are written 
in this theme, see Kitaev's survey \cite{kitaev}. 
Basic notations about pattern avoiding permutations are given in  
 Section {\rmfamily \ref{preliminaries-permutation}}. 
The most popular results on the pattern avoidance are the enumerations of 
permutations avoiding a pattern of length three, which is enumerated by 
the $Catalan \ numbers$ \cite{knuth-1} \cite{knuth-2} \cite{simion-schmidt}.
For further information on pattern avoidance, see  \cite{bona}. 
One research direction is to construct Gray codes for the permutations of given length
 avoiding a set of 
patterns and many Gray codes are found by several authors. 
We denote the set of the permutations of length $n$ avoiding a pattern $p$ by 
$S_n(p)$ and that avoiding any pattern in $A$ by $S_n (A)$, where $A$ is a set of permutations.
Juarna and Vajnovszki constructed Gray codes for 
$S_n (123,132)$ and $S_n (123,132, p (p-1) \cdots 1 (p+1))$ 
\cite{juar-vajno}. 
In \cite{dukes}, Dukes et al. gave Gray codes for large families of 
pattern avoiding permutations  including many fundamental classes.
Baril improved their results \cite{baril}. 
Their proofs are based on ECO method \cite{barcucci-lungo-pergola-pinzani}  \cite{bernini-grazzini-pergola-pinzani}. 

The most fundamental cases are  Gray codes for the permutations of length $n$ avoiding 
a single pattern of length three. 
In particular, Baril constructed Gray codes
for $S_n(p)$ for $p \in S_3$, where two consecutive  
permutations differ by  at most three positions and his results are 
 optimal for odd $n$ \cite{baril}.
Next, we should consider the permutations of length $n$
 avoiding two  patterns of length three. 
By the reversal and complementation  and their compositions, 
there are five symmetry classes avoiding two pattern of length three, 
they are $S_n(123, 321)$, $S_n(123, 231)$, 
$S_n(231, 312)$,
$S_n(123, 132)$ and $S_n(132, 312)$ \cite{simion-schmidt}.

The set $S_n(123,321)$ is empty set for $n \ge 5$. 
Also it is easy to see that 
 $S_n(123, 231)$ and $S_n(231, 312)$ have no Gray codes such that 
successive permutations differ by at most $k$ transpositions, 
where the constant $k$ does not depend on $n$. 
For instance, consider $\lfloor \frac{n}{2} \rfloor 
(\lfloor \frac{n}{2} \rfloor -1) \cdots 2 1 n (n-1) \cdots 
(\lfloor \frac{n}{2} \rfloor + 2) (\lfloor \frac{n}{2} \rfloor + 1)$, where $\lfloor \cdot \rfloor$ is the  
floor function that gives the greatest integer less than or equal to the number. 
Juarna and Vajnovszki gave a Gray code for $S_n(123, 132)$ in \cite{juar-vajno}. 
The remaining case is $S_n(132, 312)$, which is not treated in the 
previous works. 
In section {\rmfamily \ref{gray-code-permutation}}, we give a Gray code for $S_n (132,312)$ as an 
application of our Hamiltonian cycle for $G(n-1)$ in Section  {\rmfamily \ref{hamiltonian-augmentation}}.

The {\it ascent sequences}, which we define in Section {\rmfamily \ref{preliminaries-ascent-sequence}}, 
play an important role in the study of $(2+2)$-free posets and  
several researchers have found the connections  between the ascent sequences and 
other combinatorial objects enumerated by Fishburn number, see \cite{melou-claeson-dukes-kitaev}  and section 3.2.2 of \cite{kitaev}. 

One can define the patterns in the ascent sequences analogous to the patterns in permutations, 
see  Section {\rmfamily \ref{preliminaries-ascent-sequence}}.
Pattern avoidance in the ascent sequences was studied by  Baxter, 
 Duncan,  Pudwell  and 
Steingr\'imsson and 
enumerative results and further properties about the  pattern avoiding ascent 
sequences are found at  \cite{baxter-pudwell}, 
\cite{duncan-steingrimsson} and their references.

In \cite{sabri-vajnovszki}, 
Sabri and Vajnovszki constructed a Gray code for the ascent sequences of given length. 
Sabri gave Gray codes for  
$\mathcal{A}_n (p)$ for $p \in \{ 011, 101, 021, 201, 012 \} $, 
where $\mathcal{A}_n (p)$ denotes the set of ascent sequences of 
length $n$ avoiding a pattern $p$ \cite{sabri}.
His Gray codes are based on 
 the Hamming distance, the number of positions at which two sequences differ. 
In this paper, we use  more restrictive distance, we call it {\it
 the strong distance},  see Section {\rmfamily \ref{preliminaries-ascent-sequence}}, 
and construct Gray codes for 
$\mathcal{A}_n (001)$  and  $\mathcal{A}_n (010)$  
such that the strong distance of two  successive sequences is at most $2$. 

Throughout the paper, we set $\mathbb{F}_2 := \{ 0, 1 \}$ and 
 $\mathbb{F}_2^{n} := \{ \epsilon_1 \epsilon_2 \ldots \epsilon_n | 
\epsilon_i \in \mathbb{F}_2, 1 \le i \le n \}$,  the set of binary words of
 length $n$.

\section{Preliminaries}\label{preliminaries}

In this section, we give several definitions and notations which we use in this paper.

\subsection{The Hamiltonian cycles and the Gray codes}

In this section, we introduce several notations from graph theory, see Diestel's  text \cite{diestel}.
A {\it graph} is a pair $G = (V, E)$ 
 of sets such that $E \subset V \times V$. Sometimes we denote $G$ for short. 
A element of $V$ is called a {\it vertex} and that of $E$ is called an {\it edge}. 

A {\it path} is a tuple of vertices such that  consecutive vertices are 
adjacent in the graph and the vertices are all distinct. 
The {\it length} of a path is the number of the vertices minus one. 
For example, an edge is a path with two vertices and its length is  $1$.
A {\it cycle} is a path such that the starting vertex and the ending vertex 
are adjacent.
A graph is called {\it connected} 
 if there is a path connecting each pair of vertices.
For a connected  graph $G$ and two vertices $u,v \in G$, 
the {\it distance} between 
$u$ and $v$ is the length of a shortest path from $u$ to $v$ and we denote 
it by $d_{G}(u,v)$.
A {\it Hamiltonian \ path} (resp. {\it Hamiltonian \ cycle}) 
is a path (resp. cycle) that visits each vertex exactly once.
For a graph $G = (V, E)$, define $G^2$ to be the graph on $V$ such that 
two vertices are adjacent if and only if their distance in $G$ is at most  $2$.
The graph  $G^2$ is called the square of $G$.

The  combinatorial Gray codes is a mathematical model 
to generate the given combinatorial object such that each element is generated exactly once and 
any successive elements differ  in some pre-specified, usually small, way. 
It is known that 
the Gray code problem can be formulated as a Hamiltonian path or a Hamiltonian cycle problem: 
the vertices of the graph are the objects and two vertices are joined by an edge
 if they differ from each other in the pre-specified way. 

A {\it listing} of a combinatorial object is a sequence such that each element appears exactly once. 
Let $X$ and  $Y$ be  combinatorial objects with $X \cap Y = \phi$ 
 and let  $\mathbb{L}_X = (x_1, x_2, \cdots , x_m)$ 
and  $\mathbb{L}_Y = (y_1, y_2, \cdots , y_n)$  be  listings of $X$ and  $Y$
 respectively. 
Write   $\mathbb{L}_X \circ \mathbb{L}_Y
:= (x_1, x_2, \cdots , x_m, y_1, y_2, \cdots , y_n)$, 
i.e., the {\it concatenation} of $X$ and $Y$, which is a listing of $X \cup Y$.

\subsection{Permutations and pattern avoiding permutations}\label{preliminaries-permutation}

We use  one-line notation, i.e., 
we denote a permutation $\omega \in S_{n}$ by the sequence 
$\omega (1) \omega (2) \cdots \omega (n)$, where 
$S_n$ is the set of all permutations on $\{ 1, 2, \cdots , n \}$.

For two permutations $\sigma = \sigma_1 \sigma_2 \cdots \sigma_n$ and 
$\tau = \tau_1 \tau_2 \cdots \tau_n$, 
we call they  differ by an adjacent transposition 
if and only if  $ \sigma_k = \tau_k$ for $k \neq i, i+1$, 
$\sigma_i = \tau_{i+1}$ and $\sigma_{i+1} = \tau_i$ for some $1 \le i \le n-1$. In other words, $\sigma$ can be obtained 
from $\tau$ by changing adjacent positions in the one-line notation of $\tau$.
For example, $352641$ and $325641$ differ by an adjacent transposition, 
changing the second and the third positions.

Similarly, we call two permutations of same length $\sigma$ and 
$\tau$  differ by  two 
 adjacent transpositions if and only if  $\sigma \neq \tau$ and 
there exists a permutation $\theta$ such that 
$\sigma$ and 
$\theta$ differ by an adjacent transposition and also 
$\theta$ and $\tau$ do. 
For example, $352641$ and $325614$ differ by two adjacent transpositions.
We say that $\sigma$ and $\tau$ differ by at most two adjacent transpositions 
if $\sigma$ and $\tau$ differ by an adjacent transposition or two adjacent 
transpositions.
Similarly, one can also define when two permutations differ by at most $k$ 
adjacent transpositions.

For $\omega = \omega_1 \omega_2 \cdots \omega_n \in S_{n}$ and 
$ \pi = \pi_1 \pi_2 \cdots \pi_k \in S_{k}$ with $k \le n$, we say that 
 a permutation $\omega$ has a  $\pi$-$pattern$  
 if 
$st(\omega_{i_1} \omega_{i_2} \cdots \omega_{i_k}) = \pi_1 \pi_2 \cdots \pi_k$ 
for some  $1 \le i_{1} < i_{2} < \cdots < i_{k} \le n$, where 
$st(\omega_{i_1} \omega_{i_2} \cdots \omega_{i_k})$
is a permutation in $S_{k}$ defined by the following process: 
the smallest value of $\omega_{i_1} \omega_{i_2} \cdots \omega_{i_k}$ 
is replaced with $1$, the second smallest value is replaced with $2$, and 
so on. 
We call $st(\omega_{i_1} \omega_{i_2} \cdots \omega_{i_k})$ the 
$standardization$ of $\omega_{i_1} \omega_{i_2} \cdots \omega_{i_k}$.
If $st(\omega_{i_1} \omega_{i_2} \cdots \omega_{i_k}) \neq \pi_1 \pi_2 \cdots 
\pi_k$ 
for any  $1 \le i_{1} < i_{2} < \cdots < i_{k} \le n$, we say that 
$\omega$ is a $\pi$-$avoiding$ $permutation$. 
We denote the set of $\pi$- avoiding permutations of length $n$ by $S_n(\pi)$ and also 
for a set of permutations $P$, 
let $S_n(P)$ denote the set of permutations of length $n$ avoiding each pattern in $P$.

In this paper, we define  Gray codes for a set of permutations as follows.

\begin{definition}

For $n \in \mathbb{N}$ and  a subset $\mathcal{X}$ of $S_n$, a Gray code for $\mathcal{X}$ 
is a listing of $\mathcal{X}$ such that each successive permutations differ by at most $d$ 
adjacent transpositions, where $d \in \mathbb{N}$ is independent from $n$.

\end{definition}

\subsection{Ascent sequences and restricted ascent sequences}\label{preliminaries-ascent-sequence}

An {\it ascent sequence} is a sequence $x_1 x_2 \cdots x_n$ of nonnegative integers such that $x_1 = 0$ and 

\begin{eqnarray*}
x_i \le \mathbf{asc}(x_1 x_2 \cdots x_{i-1}) + 1
\end{eqnarray*}

for $2 \le i \le n$, where $\mathbf{asc}(x_1 x_2 \cdots x_{i-1}) $ is 
the number of {\it ascents} in the sequence $x_1 x_2 \cdots x_{i-1}$, that is, 
the number of $1 \le j \le i-2$ such that $x_j < x_{j+1}$.
For example, $01201014216$ is an ascent sequence of length $11$ and 
$012010635$ is not an ascent sequence, because 
 the 7th position is greater than $\mathbf{asc}(012010) + 1 = 4$.
The number of ascent sequences of length $n$ is enumerated by Fishburn number, 
see A022493 in OEIS. We denote the set of ascent sequences of length $n$ by 
$\mathcal{A}_n$.

For a nonnegative integer sequence $\boldsymbol{x} = x_1 x_2 \cdots x_n$, 
the {\it reduction} of $x$ is the the sequence obtained by replacing the 
$i$-th smallest digits of $x$ with $i-1$ and we denote it by $\mathrm{red}(\boldsymbol{x})$.
For example, $\mathrm{red}(200424) =100212$. 
One can define the patterns in the ascent sequences analogous to the patterns in  permutations.

\begin{definition}\cite{duncan-steingrimsson} 
Let $p_1 p_2 \ldots p_k$ be a sequence of nonnegative integers with 
 $\mathrm{red}(p_1 p_2 \ldots p_k) = p_1 p_2 \ldots p_k$.
An ascent sequence $a_1 a_2 \ldots a_n$ is called  a $ p_1 p_2 \ldots p_k$ pattern avoiding ascent sequence if and only if 
for any  $1 \le i_1 < i_2 < \ldots < i_k \le n$, 
$\mathrm{red}(a_{i_1} a_{i_2} \ldots  a_{i_k}) \neq  p_1 p_2 \ldots p_k$.
We denote the set of $ p_1 p_2 \ldots p_k$  avoiding ascent sequences of length $n$ by $\mathcal{A}_n( p_1 p_2 \ldots p_k)$.
\end{definition}

For a set of integer sequences $\mathcal{I}$ of given length $n$, 
a Gray code for $\mathcal{I}$ 
 is sometimes defined as a listing of $\mathcal{I}$  
 such that the {\it Hamming distance}
 between any two successive elements
(the number of positions at which two sequences differ)
 is bounded by a given constant 
which is independent from $n$.
Sabri gave Gray codes of Hamming distance $3$, 
that is a listing such that the Hamming distance of any consecutive 
elements is at most $3$, 
 for $\mathcal{A}_n(p)$, where $p \in \{ 011, 101, 021, 201 \} $ and that of Hamming distance $1$ for 
$\mathcal{A}_n(012)$ \cite{sabri}.

We define a Gray code by using another distance. 
 For  $a_1 a_2 \ldots a_n, b_1 b_2 \ldots b_n \in \mathcal{A}_n$, 
define  

\begin{displaymath}
d_{\mathrm{str}} (a_1 a_2 \ldots a_n, b_1 b_2 \ldots b_n) := \sum_{k=1}^{n} |a_k - b_k|, 
\end{displaymath} 

where $| \cdot |$ is the absolute value of given number. 

We call $d_{\rm{str}}$ {\it strong distance}. 
The strong distance is essentially equivalent to 
the  {\it adjacent interchange property} which is referred in 
\cite{savage}. 
In analogy with Sabri's definition in \cite{sabri}, 
we can state the following definition.

\begin{definition}

For a subset $\mathcal{X}$ of 
$\mathcal{A}_n$, a Gray code of strong distance
 $d$ for $\mathcal{X}$ is a listing of  $\mathcal{X}$ 
such that  each strong distance of two successive elements is
at most  $d$ and there exists two successive elements of strong 
distance $d$, where $d \in \mathbb{N}$ is independent from $n$.
\end{definition}

The strong distance of given two ascent sequences is larger than or equal to the Hamming distance of them and the strong distance  is a more restrictive distance than the Hamming distance.

\section{A Hamiltonian cycle for the square of 
the augmentation graph}\label{hamiltonian-augmentation}

Douglas West defined the augmentation graph of size $n$ on the 
subsets of $\{ 1, 2, \ldots , n \}$ \cite{savage}.
We realize the augmentation graph on $\mathbb{F}_2^{n}$ 
 and we denote it by $G(n)$.
No confusion should results when we 
call our graph $G(n)$ the augmentation graph of size $n$.

Set  
$V(G(n)) :=
 \mathbb{F}_2^{n}$  
and draw an edge between 
$ \epsilon_{1} \epsilon_{2} \cdots \epsilon_{n}$ and 
$\epsilon^{\prime}_{1} \epsilon^{\prime}_{2} \cdots \epsilon^{\prime}_{n} \in \mathbb{F}_2^{n}$
if 

\begin{enumerate}

\item $\epsilon_{1} \neq \epsilon^{\prime}_{1}$ and 
$\epsilon_{i} = \epsilon^{\prime}_{i}$ for $2 \le i \le n$ or 

\item for some $1 \le i \le n-1$, 
$\epsilon_i \epsilon_{i+1} = 01$ and $\epsilon^{\prime}_i 
\epsilon^{\prime}_{i+1} = 10$ or 
$\epsilon_i \epsilon_{i+1} = 10$ and $\epsilon^{\prime}_i 
\epsilon^{\prime}_{i+1} = 01$ 
and $\epsilon_j = \epsilon^{\prime}_j$ for $j \neq i, i+1$.

In other words,  two vertices are adjacent 
 if and only if 
 they only differ in the first positions or 
  differ   by an interchange of a $0$ and a $1$
 in adjacent positions.

\end{enumerate}

\begin{remark}

In $G(n)$, the degree of  $00 \cdots 0$ and $11 \cdots 1$ is $1$.  
Hence $G(n)$  has no Hamiltonian cycles for $ n \ge 2$. 
To the best of our knowledge,  there are no Hamiltonian paths for  $G(n)$ with $\frac{n(n-1)}{2}$ being 
even and the existence 
of Hamiltonian paths  is
 not known for  $n \ge 7$ with 
 $\frac{n(n-1)}{2}$ being odd \cite{savage}.

\end{remark}

Figure \ref{fig:g-3-graph} shows   $G(3)$, left hand side, and $G^2(3)$, right hand side.

\begin{figure}[htbp]
\begin{center}
\includegraphics[width=100mm]{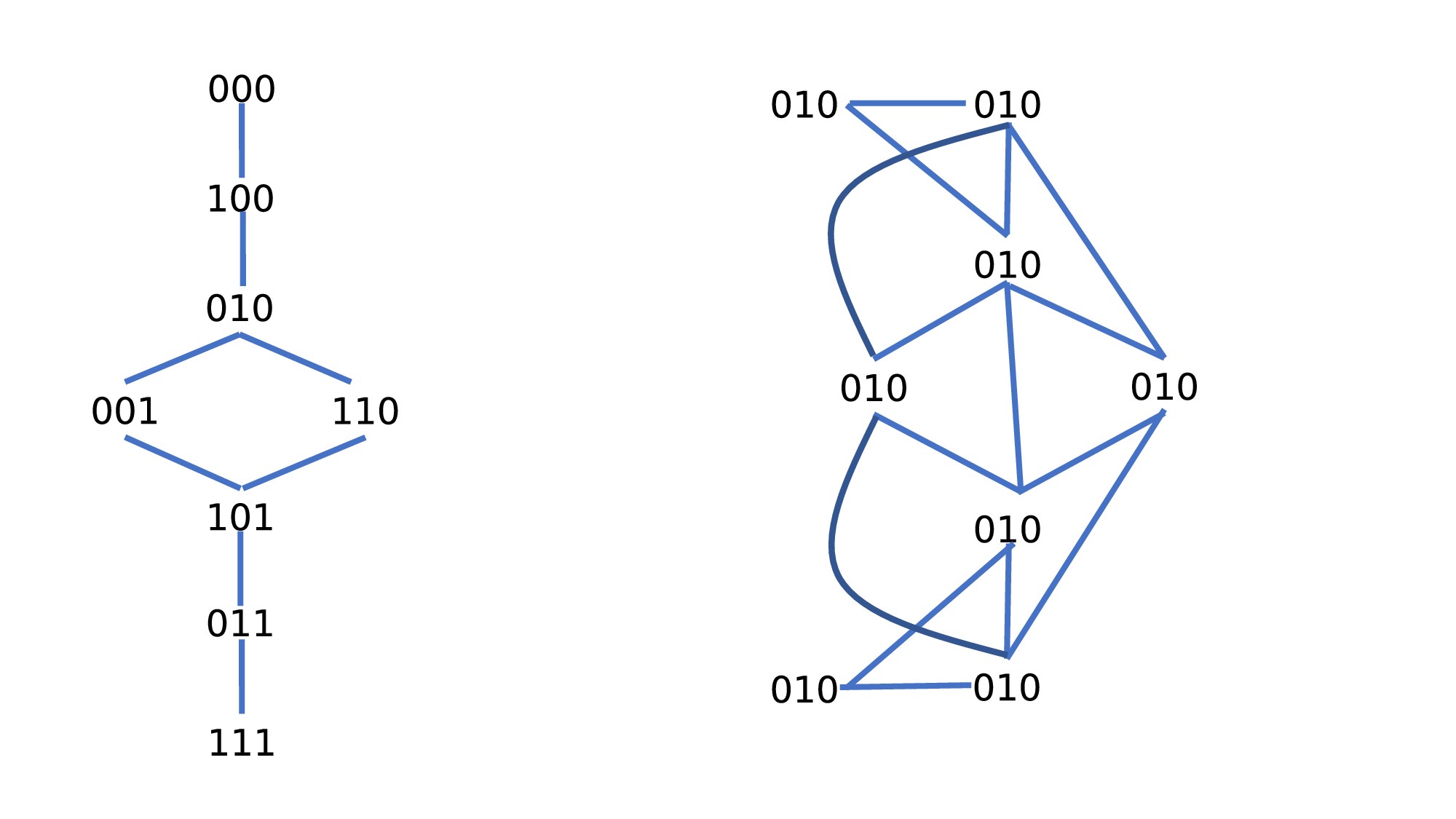}
\caption{ The graphs $G(3)$ and $G^2(3)$. }
\label{fig:g-3-graph}
\end{center}
\end{figure}

\begin{remark}

A listing of $V(G(n))$ is a $2^n$-tuple 
of  $V(G(n))$  such that each vertex appears exactly once.
For a listing  
$(\alpha_1, \alpha_2, \cdots , \alpha_{2^n})$ such that every distance of 
successive vertices is at most $2$, 
  we say that distance $2$ jumps do not appear consecutively in the listing 
if there is no $2 \le i \le 2^{n} - 1$ such that  
$d_{G(n)}(\boldsymbol{\alpha}_{i-1}, \boldsymbol{\alpha}_{i}) = 
d_{G(n)}(\boldsymbol{\alpha}_{i}, \boldsymbol{\alpha}_{i+1}) = 2$.

\end{remark}

\begin{definition}

For $\boldsymbol{\epsilon} = \epsilon_1 \epsilon_2 \cdots \epsilon_m 
\in \mathbb{F}_{2}^{m}$ and $\eta, \mu \in \mathbb{F}_{2}$, define  
$\boldsymbol{\epsilon}  \eta := 
 \epsilon_1 \epsilon_2 \cdots \epsilon_m  \eta $
(resp. $\boldsymbol{\epsilon}  \eta \mu := 
\epsilon_1 \epsilon_2 \cdots \epsilon_m  \eta \mu$ )
 which 
is the concatenation of $\boldsymbol{\epsilon}$ and $\eta$, 
(resp. $\boldsymbol{\epsilon}$, $\eta$ and $\mu$) and 
it is an element of $\mathbb{F}_{2}^{m+1}$ (resp. $\mathbb{F}_{2}^{m+2}$).

\end{definition}

The following  is the main result of this section.

\begin{theorem}\label{augment-hamilton-cycle}
For $n \ge 2$, there is a listing  

\begin{equation}
( \boldsymbol{\epsilon}_{1}, \boldsymbol{\epsilon}_{2},  
\cdots , \boldsymbol{\epsilon}_{2^{n}} ) 
\end{equation}

of $V(G(n))$, 
which satisfies: \\

{\rm (L1)}
 $\boldsymbol{\epsilon}_{1} = 00 \cdots 0$, 
all entries are $0$, and
$ \boldsymbol{\epsilon}_{2^{n}} = 100 \cdots 0$, 
the  first entry is $1$ and the rest of the
 entries  are all $0$;

{\rm (L2)}
 $d_{G(n)}(\boldsymbol{\epsilon}_{i}, \boldsymbol{\epsilon}_{i+1}) \le 2$ 
for $1 \le i \le 2^{n} - 1$;

{\rm (L3)}
if $d_{G(n)}(\boldsymbol{\epsilon}_{i}, \boldsymbol{\epsilon}_{i+1}) = 2$,  then  $d_{G(n)}(\boldsymbol{\epsilon}_{i-1}, \boldsymbol{\epsilon}_{i}) = 
d_{G(n)}(\boldsymbol{\epsilon}_{i+1}, \boldsymbol{\epsilon}_{i+2}) = 1$ for 
$2 \le i \le 2^{n} - 2$.

In other words, there is a listing which starts from $ 00 \cdots 0$ and ends at  $100 \cdots 0$ such that every distance of successive elements is at most $2$ and distance $2$ jumps do not appear consecutively.

\end{theorem}

In the above listing $( \boldsymbol{\epsilon}_{1}, \boldsymbol{\epsilon}_{2},  \cdots , \boldsymbol{\epsilon}_{2^{n}} )$, the distance of 
$\boldsymbol{\epsilon}_{1}$ and $\boldsymbol{\epsilon}_{2^n}$ is 
$1$. Hence we obtain the following Corollary.

\begin{corollary}\label{main-augmentation}

For $n \in \mathbb{N}_{\ge 2}$, 
the square of $G(n)$ has a Hamiltonian cycle.

\end{corollary}

To prove Theorem {\rmfamily \ref{augment-hamilton-cycle}}, 
we show Proposition {\rmfamily \ref{augmentation-prop-odd-to-even}} and Proposition {\rmfamily \ref{augmentation-prop-even-to-odd}}.

\begin{proposition}\label{augmentation-prop-odd-to-even}

For $k \ge 2$, suppose that there is a listing 

\begin{equation}
 ( \boldsymbol{\epsilon}_{1}, \boldsymbol{\epsilon}_{2},  
\cdots , \boldsymbol{\epsilon}_{2^{2k-1}} )
\end{equation}

of $V(G(2k-1))$, which satisfies: \\

{\rm (A1)}
  $\boldsymbol{\epsilon}_{1} = 00 \cdots 00$,  all entries are $0$, and 
$\boldsymbol{\epsilon}_{2^{2k-1}} = 10 \cdots 00$, the 
first entry is $1$ and the rest of the entries  are all $0$; 

{\rm (A2)}  
 $d_{G(2k-1)}(\boldsymbol{\epsilon}_{i}, \boldsymbol{\epsilon}_{i+1}) \le 2$ 
for $1 \le i \le 2^{2k-1} - 1$ and distance $2$ jumps do not appear consecutively;

{\rm (A3)}  the element $00 \cdots 01$, the last entry is $1$ and the rest of the 
entries are all $0$, appears next to $10 \cdots 01$, the first and the last entries are $1$ and the rest of the entries are all $0$; 

{\rm (A4)}  $\boldsymbol{\epsilon}_{2^{2k-1}-1}$, the second from the last 
entry of the listing, is $001 \cdots 00$, the third entry is $1$ and the rest of entries are all $0$.

Then there is a listing  

\begin{equation}
( \boldsymbol{\eta}_{1}, \boldsymbol{\eta}_{2},  
\cdots , \boldsymbol{\eta}_{2^{2k}} )
\end{equation}

of $V(G(2k))$, 
which satisfies: \\

{\rm (B1)}
 $\boldsymbol{\eta}_{1} = 00 \cdots 00$, all entries are $0$, and
$\boldsymbol{\eta}_{2^{2k}} = 10 \cdots 00$, 
the first entry is $1$ and the rest of the entries  are all $0$; 

{\rm (B2)}  
 $d_{G(2k)}(\boldsymbol{\eta}_{i}, \boldsymbol{\eta}_{i+1}) \le 2$ for 
 $1 \le i \le 2^{2k} - 1$
and distance $2$ jumps do not appear consecutively; 

{\rm (B3)}  
$001 \cdots 01$, the third and the last entries are $1$ and the rest of the entries are all $0$, appears next to $100 \cdots 01$, the first and the last entries are $1$ and the rest of the entries are all $0$;

{\rm (B4)}
   $\boldsymbol{\eta}_{2^{2k}-1}$, the second from the last 
entry in the listing, is $001 \cdots 0$, the third entry is $1$ and the rest of the entries are all $0$.

\end{proposition}

\begin{proof}

Let $\mathbb{L}_{2k-1} = ( \boldsymbol{\epsilon}_{1}, \boldsymbol{\epsilon}_{2},\cdots , \boldsymbol{\epsilon}_{2^{2k-1}} )$ be a listing of $V(G(2k-1))$ 
which satisfies the  conditions from {\rm (A1)} to {\rm (A4)}. 
From {\rm (A3)},  we have 
$\boldsymbol{\epsilon}_i =  10 \cdots 01$, the first and the 
last entries are $1$ 
and the rest of the entries are all $0$,
 and 
$\boldsymbol{\epsilon}_{i+1} =  00 \cdots 01$, the last entry is $1$ and 
the rest of entries are all $0$, for some $2 \le i \le 2^{2k-1}-2$.
Remark that $\boldsymbol{\epsilon}_{2^{2k-1} - 1} = 001 \cdots 00$, the third 
entry is $1$ and the rest of entries are all $0$.

Set 

\begin{equation}
\mathbb{X}_1 := (\boldsymbol{\epsilon}_{1}  0, 
\boldsymbol{\epsilon}_{2}  0, \cdots , 
\boldsymbol{\epsilon}_{i}  0),
\end{equation}

\begin{equation}
\mathbb{X}_2 := (\boldsymbol{\epsilon}_{2^{2k-1}} 1, 
\boldsymbol{\epsilon}_{2^{2k-1} -1}  1, \cdots , 
\boldsymbol{\epsilon}_{2}  1, 
\boldsymbol{\epsilon}_{1}  1),
\end{equation}

\begin{equation}
\mathbb{X}_3 := (\boldsymbol{\epsilon}_{i+1}  0, 
\boldsymbol{\epsilon}_{i+2}  0, \cdots , 
\boldsymbol{\epsilon}_{2^{2k-1} -1} 0, 
\boldsymbol{\epsilon}_{2^{2k-1}}  0).
\end{equation}

The concatenation 
 $ \mathbb{X}_1 \circ \mathbb{X}_2 \circ \mathbb{X}_3$ 
is  a listing of $G(2k)$ and put 
$ \mathbb{X}_1 \circ \mathbb{X}_2 \circ \mathbb{X}_3 = ( \boldsymbol{\eta}_{1}, \boldsymbol{\eta}_{2},  
\cdots , \boldsymbol{\eta}_{2^{2k}} )$. Remark that $|\mathbb{X}_3| \ge 2$.

From the construction, we have $\boldsymbol{\eta}_{1} = 00 \cdots 00$ and 
$\boldsymbol{\eta}_{2^{2k}} = 10 \cdots 00$.  We see {\rm (B1)}.

For each $\mathbb{X}_1$, $\mathbb{X}_2$ and $\mathbb{X}_3$, 
every distance of successive elements is at most $2$ and 
distance $2$ jumps do not appear consecutively. 
Also, the distance between the last element of $\mathbb{X}_1$ and 
the first element of $\mathbb{X}_2$ is $1$ and that of 
 between the last element of $\mathbb{X}_2$ and 
the first element of $\mathbb{X}_3$ is   $1$.
 Hence   $ \mathbb{X}_1 \circ \mathbb{X}_2 \circ \mathbb{X}_3$  satisfies {\rm (B2)}.

The first entry of $\mathbb{X}_2$ is 
$100 \cdots 01$, the first and the last entries are $1$ and the rest of entries are all $0$.
Because $\boldsymbol{\epsilon}_{2^{2k-1}-1} = 001 \cdots 00$, the third entry is $1$ and the rest of entries are all $0$, 
the second entry of $\mathbb{X}_2$ is $001 \cdots 01$, the third and the
 last entries are $1$ and the rest of entries are all $0$. 
Hence we see {\rm (B3)}.

The second from the last entry of 
$ \mathbb{X}_1 \circ \mathbb{X}_2 \circ \mathbb{X}_3 $
is in $\mathbb{X}_3$, because  $|\mathbb{X}_3| \ge 2$, and 
that is $\boldsymbol{\epsilon}_{2^{2k-1}-1}  0 = 001 \cdots 00$, the third 
entry is $1$ and the rest of the entries are all $0$. Hence we see  {\rm (B4)}. 
This completes the proof.

\end{proof}

\begin{example}\label{01-listing-n=4-example}

For $G(3)$, put 

\begin{center}
\begin{tabular}{cccccccc}

$\boldsymbol{\epsilon}_1 = 000$, & 
$\boldsymbol{\epsilon}_2 = 010$, &
$\boldsymbol{\epsilon}_3 = 110$, &
$\boldsymbol{\epsilon}_4 = 011$, &
$\boldsymbol{\epsilon}_5 = 111$, &
$\boldsymbol{\epsilon}_6 = 101$, &
$\boldsymbol{\epsilon}_7 = 001$, &
$\boldsymbol{\epsilon}_8 = 100$ \\
\end{tabular}

\end{center}

and then, the listing 
$(\boldsymbol{\epsilon}_1, \boldsymbol{\epsilon}_2, \cdots 
,  \boldsymbol{\epsilon}_8)$ satisfies 
from {\rm (A1)} to {\rm (A4)}. 
Set 

\begin{equation}
\mathbb{X}_1 := 
(\boldsymbol{\epsilon}_1  0, \boldsymbol{\epsilon}_2  0, 
\boldsymbol{\epsilon}_3  0, \boldsymbol{\epsilon}_4  0,
\boldsymbol{\epsilon}_5  0, \boldsymbol{\epsilon}_6  0 ),
\end{equation}

\begin{equation}
\mathbb{X}_2 := 
(\boldsymbol{\epsilon}_8  1, \boldsymbol{\epsilon}_7  1, 
\boldsymbol{\epsilon}_6  1, \boldsymbol{\epsilon}_5  1,
\boldsymbol{\epsilon}_4  1, \boldsymbol{\epsilon}_3  1, 
\boldsymbol{\epsilon}_2  1, \boldsymbol{\epsilon}_1  1 )
\end{equation}

and 

\begin{equation}
\mathbb{X}_3 := 
(\boldsymbol{\epsilon}_7  0, \boldsymbol{\epsilon}_8  0).
\end{equation}

Then the  listing $\mathbb{X}_1 \circ \mathbb{X}_2  \circ 
\mathbb{X}_3$ is

\begin{equation}
(0000, 0100, 1100, 0110, 1110, 1010, 
1001, 0011, 1011, 1111, 0111, 1101, 0101, 0001, 
0010, 1000)
\end{equation}

and  that satisfies from {\rm (B1) } to {\rm (B4)}.

\end{example}

\begin{proposition}\label{augmentation-prop-even-to-odd}

For $k \ge 2$, suppose that there is a listing

\begin{equation}
 ( \boldsymbol{\epsilon}_{1}, \boldsymbol{\epsilon}_{2},  
\cdots , \boldsymbol{\epsilon}_{2^{2k}} )
\end{equation}

 of $V(G(2k))$, 
which satisfies: \\

{\rm (B1)}
 $\boldsymbol{\epsilon}_{1} = 00 \cdots 00$, all entries are $0$, and
$\boldsymbol{\epsilon}_{2^{2k}} = 10 \cdots 00$, 
the first entry is $1$ and the rest of the entries  are all $0$; 

{\rm (B2)}  
  $d_{G(2k)}(\boldsymbol{\epsilon}_{i}, \boldsymbol{\epsilon}_{i+1}) \le 2$ 
for $1 \le i \le 2^{2k} - 1$
 and distance $2$ jumps do not appear consecutively; 

{\rm (B3)}  
$001 \cdots 01$, the third and the last entries are $1$ and the rest of the entries are all $0$, appears next to $100 \cdots 01$, the first and the last entries are $1$ and the rest of the entries are all $0$;

{\rm (B4)}
   $\boldsymbol{\epsilon}_{2^{2k}-1}$, the second from the last 
entry in the listing, is $001 \cdots 00$, the third entry is $1$ and the rest of the entries are all $0$.

Then there is a listing

\begin{equation}
( \boldsymbol{\eta}_{1}, \boldsymbol{\eta}_{2},  
\cdots , \boldsymbol{\eta}_{2^{2k+1}} )
\end{equation}

 of $V(G(2k+1))$, 
which satisfies:

{\rm (A1)}
  $\boldsymbol{\eta}_{1} = 00 \cdots 00$,  all entries are $0$, and 
$\boldsymbol{\eta}_{2^{2k+1}} = 10 \cdots 00$, the first entry is $1$ and the rest of the
 entries  are all $0$; 

{\rm (A2)}  
 $d_{G(2k+1)}(\boldsymbol{\eta}_{i}, \boldsymbol{\eta}_{i+1}) \le 2$ 
for $1 \le i \le 2^{2k} - 1$
 and distance $2$ jumps do not appear consecutively; 

{\rm (A3)}  the element $00 \cdots 01$, the last entry is $1$ and the rest of 
the entries are all $0$, appears next to $10 \cdots 01$, the first and the last entries are $1$ and the rest of the entries are all $0$; 

{\rm (A4)}  $\boldsymbol{\eta}_{2^{2k+1}-1}$, the second from the last 
entry in the listing, is $001 \cdots 00$, the third entry is $1$ and the rest of entries are all $0$.

\end{proposition}

\begin{proof}

Let $\mathbb{L}_{2k} = ( \boldsymbol{\epsilon}_{1}, \boldsymbol{\epsilon}_{2},\cdots , \boldsymbol{\epsilon}_{2^{2k}} )$ be a listing of $V(G(2k))$ 
which satisfies the  conditions from {\rm (B1)} to {\rm (B4)}. 
From {\rm (B3)},  we have 
$\boldsymbol{\epsilon}_i =  10 \cdots 01$, the first and the 
last entries are $1$ 
and the rest of entries are all $0$,
 and 
$\boldsymbol{\epsilon}_{i+1} =  001 \cdots 01$, the third and the last entries 
are $1$ and 
the rest of entries are all $0$, for some $2 \le i \le 2^{2k}-2$.
Remark that $\boldsymbol{\epsilon}_{2^{2k} - 1} = 001 \cdots 00$, the third 
entry is $1$ and the rest of the entries are all $0$.

Set

\begin{equation}
\mathbb{Y}_1 := (\boldsymbol{\epsilon}_{1}  0, 
\boldsymbol{\epsilon}_{2}  0, \cdots , 
\boldsymbol{\epsilon}_{i}  0),
\end{equation}

\begin{equation}
\mathbb{Y}_2 := (\boldsymbol{\epsilon}_{2^{2k}}  1, 
\boldsymbol{\epsilon}_{1}  1, 
\boldsymbol{\epsilon}_{2}  1, \cdots , 
\boldsymbol{\epsilon}_{2^{2k} - 2}  1, 
\boldsymbol{\epsilon}_{2^{2k} - 1}  1),
\end{equation}

\begin{equation}
\mathbb{Y}_3 := (\boldsymbol{\epsilon}_{i+1}  0, 
\boldsymbol{\epsilon}_{i+2}  0, \cdots , 
\boldsymbol{\epsilon}_{2^{2k}}  0).
\end{equation}

The concatenation 
 $ \mathbb{Y}_1 \circ \mathbb{Y}_2 \circ \mathbb{Y}_3$ 
is  a listing of $G(2k+1)$ and put 
 $ \mathbb{Y}_1 \circ \mathbb{Y}_2 \circ \mathbb{Y}_3 = 
( \boldsymbol{\eta}_{1}, \boldsymbol{\eta}_{2},  
\cdots , \boldsymbol{\eta}_{2^{2k+1}} )$. 
Remark that $|\mathbb{Y}_3| \ge 2$.

From the construction, we have 
$\boldsymbol{\eta}_{1} = 00 \cdots 00$,  all entries are $0$ and 
$\boldsymbol{\eta}_{2^{2k+1}} = 10 \cdots 00$, first entry is $1$ and the rest of the entries  are all $0$. 
Hence we see  (A1).

In  $\mathbb{Y}_2$, the distance between the first element and 
the second element is $1$. Hence, for each 
 $\mathbb{Y}_1$, $\mathbb{Y}_2$ and $\mathbb{Y}_3$, 
every distance of successive elements is at most $2$ and 
distance $2$ jumps do not appear consecutively.
Also the distance between the last element of $\mathbb{Y}_1$ and the first 
element of $\mathbb{Y}_2$ is $1$ and  that of 
between the last element of $\mathbb{Y}_2$ and the first element of 
$\mathbb{Y}_3$ is $1$. Hence   $ \mathbb{Y}_1 \circ \mathbb{Y}_2 \circ \mathbb{Y}_3$  satisfies (A2).

The first element of $\mathbb{Y}_2$ is $10 \cdots 01$, the first and the 
last entries are $1$ and the rest of entries are all $0$, and the second element of 
$\mathbb{Y}_2$ is $00 \cdots 01$, the last entry is $1$ and the rest of 
entries are all $0$. Hence we obtain  (A3).
The second from the last entry of  $ \mathbb{Y}_1 \circ \mathbb{Y}_2 \circ \mathbb{Y}_3$ 
is in $\mathbb{Y}_3$, because $|\mathbb{Y}_3| \ge 2$, and that is $\boldsymbol{\epsilon}_{2^{2k}-1} 0 = 
001 \cdots 00$. Hence we have (A4). 
This completes the proof.

\end{proof}

\begin{example}\label{01-listing-n=5-example}

From example {\rmfamily \ref{01-listing-n=4-example}}, put

\begin{center}
\begin{tabular}{ccccccc}
$\boldsymbol{\epsilon}_1 = 0000$, &
$\boldsymbol{\epsilon}_2 = 0100$, &
$\boldsymbol{\epsilon}_3 = 1100$, &
$\boldsymbol{\epsilon}_4 = 0110$, &
$\boldsymbol{\epsilon}_5 = 1110$, &
$\boldsymbol{\epsilon}_6 = 1010$, &
$\boldsymbol{\epsilon}_7 = 1001$, \\
$\boldsymbol{\epsilon}_8 = 0011$, & 
$\boldsymbol{\epsilon}_9 = 1011$, & 
$\boldsymbol{\epsilon}_{10} = 1111$, &
$\boldsymbol{\epsilon}_{11} = 0111$, &
$\boldsymbol{\epsilon}_{12} = 1101$, &
$\boldsymbol{\epsilon}_{13} = 0101$, &
$\boldsymbol{\epsilon}_{14} = 0001$, \\
$\boldsymbol{\epsilon}_{15} = 0010$, &
$\boldsymbol{\epsilon}_{16} = 1000$. \\

\end{tabular}
\end{center}

The  listing
$(\boldsymbol{\epsilon}_1, \boldsymbol{\epsilon}_2, \cdots 
, \boldsymbol{\epsilon}_{16})$  of $V(G(4))$ satisfies from {\rm (B1)} to {\rm (B4)}. 

Set

\begin{equation}
\mathbb{Y}_1 := 
(\boldsymbol{\epsilon}_1  0, \boldsymbol{\epsilon}_2  0, 
\boldsymbol{\epsilon}_3  0, \boldsymbol{\epsilon}_4  0,
\boldsymbol{\epsilon}_5  0, \boldsymbol{\epsilon}_6  0, 
\boldsymbol{\epsilon}_7  0 ),
\end{equation}

\begin{equation}
\mathbb{Y}_2 := 
(\boldsymbol{\epsilon}_{16}  1, \boldsymbol{\epsilon}_1  1, 
\boldsymbol{\epsilon}_2  1, \boldsymbol{\epsilon}_3  1,
\boldsymbol{\epsilon}_4  1, \boldsymbol{\epsilon}_5  1, 
\boldsymbol{\epsilon}_6  1, \boldsymbol{\epsilon}_7  1, 
\boldsymbol{\epsilon}_8  1, \boldsymbol{\epsilon}_9  1,
\boldsymbol{\epsilon}_{10}  1, \boldsymbol{\epsilon}_{11}  1, 
\boldsymbol{\epsilon}_{12}  1, \boldsymbol{\epsilon}_{13}  1, 
\boldsymbol{\epsilon}_{14}  1, \boldsymbol{\epsilon}_{15}  1 )
\end{equation}

and 

\begin{equation}
\mathbb{Y}_3 := 
(\boldsymbol{\epsilon}_8  0, \boldsymbol{\epsilon}_9  0
\boldsymbol{\epsilon}_{10}  0, \boldsymbol{\epsilon}_{11}  0, 
\boldsymbol{\epsilon}_{12}  0, \boldsymbol{\epsilon}_{13}  0, 
\boldsymbol{\epsilon}_{14}  0, \boldsymbol{\epsilon}_{15} 0
\boldsymbol{\epsilon}_{16}  0 ).
\end{equation}

Then the  listing 
 $\mathbb{Y}_1 \circ \mathbb{Y}_2 \circ 
\mathbb{Y}_3 = 
(\boldsymbol{\eta}_1, \boldsymbol{\eta}_2, \ldots , \boldsymbol{\eta}_{32})$ of $V(G(5))$ is

\begin{center}
\begin{tabular}{ccccccc}
$\boldsymbol{\eta}_1 = 00000$, &
$\boldsymbol{\eta}_2 = 01000$, &
$\boldsymbol{\eta}_3 = 11000$, &
$\boldsymbol{\eta}_4 = 01100$, &
$\boldsymbol{\eta}_5 = 11100$, &
$\boldsymbol{\eta}_6 = 10100$, &
$\boldsymbol{\eta}_7 = 10010$, \\
$\boldsymbol{\eta}_{8} = 10001$, &
$\boldsymbol{\eta}_{9} = 00001$, &
$\boldsymbol{\eta}_{10} = 01001$, &
$\boldsymbol{\eta}_{11} = 11001$, &
$\boldsymbol{\eta}_{12} = 01101$, &
$\boldsymbol{\eta}_{13} = 11101$, &
$\boldsymbol{\eta}_{14} = 10101$, \\
$\boldsymbol{\eta}_{15} = 10011$, &
$\boldsymbol{\eta}_{16} = 00111$, & 
$\boldsymbol{\eta}_{17} = 10111$, & 
$\boldsymbol{\eta}_{18} = 11111$, &
$\boldsymbol{\eta}_{19} = 01111$, &
$\boldsymbol{\eta}_{20} = 11011$, &
$\boldsymbol{\eta}_{21} = 01011$, \\
$\boldsymbol{\eta}_{22} = 00011$, &
$\boldsymbol{\eta}_{23} = 00101$, &
$\boldsymbol{\eta}_{24} = 00110$, & 
$\boldsymbol{\eta}_{25} = 10110$, & 
$\boldsymbol{\eta}_{26} = 11110$, &
$\boldsymbol{\eta}_{27} = 01110$, &
$\boldsymbol{\eta}_{28} = 11010$, \\
$\boldsymbol{\eta}_{29} = 01010$, &
$\boldsymbol{\eta}_{30} = 00010$, &
$\boldsymbol{\eta}_{31} = 00100$, &
$\boldsymbol{\eta}_{32} = 10000$, \\

\end{tabular}
\end{center}

and it satisfies from {\rm (A1)} to {\rm (A4)}.

\end{example}

\begin{proof}(Proof of Theorem {\rmfamily \ref{augment-hamilton-cycle}}.)

Put $\mathbb{L}_2 := (00, 01, 11, 10)$ and 
$\mathbb{L}_3 := ( 000, 010, 110, 011, 111, 101, 001,  100)$ and they 
 satisfy {\rm (L1)}, {\rm (L2)} and {\rm (L3)}.

From Proposition {\rmfamily \ref{augmentation-prop-odd-to-even}},  
Proposition {\rmfamily \ref{augmentation-prop-even-to-odd}} and the inductive method, 
we can construct 
a listing $\mathbb{L}_{2k}$ for $V(G(2k))$ 
(resp. $\mathbb{L}_{2k+1}$ for $V(G(2k+1))$), 
which satisfies from {\rm (B1) } to {\rm (B4)} 
(resp. from {\rm (A1) } to {\rm (A4)}) for $k \ge 2$. 
Our  listings  $\mathbb{L}_{2k}$ and $\mathbb{L}_{2k+1}$  satisfy 
{\rm (L1)}, {\rm (L2)} and {\rm (L3)} for $k \ge 2$. 
Therefore we obtain  the Hamiltonian cycles for the square of the augmentation graphs.

\end{proof}

\begin{remark}\label{binary-dist-is-two}
Our listing in Theorem {\rmfamily \ref{augment-hamilton-cycle}} is not a Hamiltonian cycle, because the distance between the first element and the second element is $2$. 

\end{remark}

\section{A Gray code for $S_n (132, 312)$}\label{gray-code-permutation}

The $132$-$312$ avoiding permutations is called Gilbreath permutations. 
A permutation $\sigma = \sigma_1 \sigma_2 \ldots \sigma_n \in S_n$ with 
$\sigma_1 = k$, where $1 \le k \le n$,  is a  
$132$-$312$ avoiding permutation if and only if 
$\sigma$ is a shuffle of the sequences $k (k-1) \ldots 2 1$ and 
$(k+1) (k+2) \ldots (n-1) n$ with $\sigma_1 = k$, see Proposition 12 in \cite{simion-schmidt}.
For example, $5{\bf 67}43{\bf 8}2{\bf 9}1 \in S_{9}$ 
is a shuffle of $54321$ and ${\bf 6789}$ and 
it is a  $132$-$312$ avoiding permutation. 
In this section, we construct a Gray code for $S_n(132, 312)$ for
 $n \ge 2$.

First we construct a natural bijection from $V(G(n-1))$ to $S_n(132, 312)$ which 
preserves minimal changes of these objects. 
We note that our map is essentially the same bijection presented at the proof of Proposition 12 in 
\cite{simion-schmidt}. 

\begin{definition}\label{def-bin-to-132-312}

For $n \in \mathbb{N}_{\ge 2}$ and 
$\boldsymbol{\epsilon} = 
\epsilon_1 \epsilon_2 \ldots \epsilon_{n-1} \in V(G(n-1))$, 
define $\Psi_n( \boldsymbol{\epsilon})$ to be a positive integer sequence of length $n$, say 
$a_1 a_2 \ldots a_n$,    such that

{\rm (1)} $a_1$ is the cardinality  of $0$ in $\boldsymbol{\epsilon}$
 plus $1$,

{\rm (2)} \ if $\epsilon_{i-1} = 0$ with $i \ge 2$, 
then $a_i$ is the cardinality of $j$ such that 
$j \ge i-1$ and $\epsilon_j = 0$, 

{\rm (3)} \ if $\epsilon_{i-1} = 1$ with $i \ge 2$, 
then $a_i$ is $a_1$ plus the cardinality of $j$ 
such that $j \le i-1$ and $\epsilon_j = 1$.

\end{definition}

For example, $\Psi(1001011) = 45326178$ and 
$\Psi(0001011) = 54326178$.
It is easy to see the following Lemma. 

\begin{lemma}
Notation is as above, $a_1 a_2 \ldots a_n$ is a permutation of length $n$. Furthermore, 
it is a shuffle of 
$a_1 (a_1-1) \ldots 21$ and $(a_1 + 1) (a_1 + 2) \ldots (n-1) n$ and 
  hence it is a  $132$-$312$ avoiding permutation. 
Moreover,  $\Psi_n$ is a bijection from $G(n-1)$ to $S_n (132,312)$.

\end{lemma}

\begin{proposition}

If $\boldsymbol{\epsilon} = \epsilon_1 \epsilon_2 \cdots \epsilon_{n-1}$ and 
$\boldsymbol{\eta} = \eta_1 \eta_2 \cdots \eta_{n-1}$ are adjacent in 
$G(n-1)$, then the corresponding permutations 
$\Psi_n (\boldsymbol{\epsilon})$ and 
$\Psi_n (\boldsymbol{\eta})$ differ by an adjacent transposition.
\end{proposition}

\begin{proof}

By assumption, $\boldsymbol{\eta}$ can be obtained from
 $\boldsymbol{\epsilon}$ by interchanging adjacent $0$ and $1$ or 
by changing the first position of $\boldsymbol{\epsilon}$.

\underline{Case 1.} 
We discuss the case when $\boldsymbol{\eta}$ is obtained by interchanging adjacent $0$ and $1$ in $\boldsymbol{\epsilon}$.
Without loss of generality, we can assume that 
$\eta_x= \epsilon_x$ for $x\neq i-1, i$, 
$\epsilon_{i-1} = \eta_{i} = 1$ and 
$\epsilon_{i} = \eta_{i-1} = 0$ for some $1 \le i \le n-1$, i.e.,

\begin{eqnarray*}
\boldsymbol{\epsilon} &=& \epsilon_1 \epsilon_2 \ldots
\epsilon_{i-2} {\bf 1} {\bf 0} 
  \epsilon_{i+1} \ldots \epsilon_{n-1},  \\
\boldsymbol{\eta} &=& \epsilon_1 \epsilon_2 \ldots 
\epsilon_{i-2}  {\bf 0} {\bf 1}  \epsilon_{i+1} \ldots \epsilon_{n-1}.
\end{eqnarray*}

Set $\Psi_n (\boldsymbol{\epsilon}) := a_1 a_2 \ldots a_n$. 
If $\Psi_n (\boldsymbol{\eta}) = b_1 b_2 \ldots b_n$, then 
by Definition
{\rmfamily \ref{def-bin-to-132-312}}, we see 
$a_x = b_x$ for $x \neq i, i+1$, 
$a_i = b_{i+1}$ and $a_{i+1} = b_i$. Hence 
$\Psi_n (\boldsymbol{\epsilon})$ and 
$\Psi_n (\boldsymbol{\eta})$ differ by an adjacent transposition.

\underline{Case 2.} We discuss the case when $\boldsymbol{\eta}$ is obtained 
by changing the first position of  $\boldsymbol{\epsilon}$.
Without loss of generality, we can assume that 
$\eta_i= \epsilon_i$ for $ 2 \le i \le n-1$ and  
$\epsilon_{1} = 0$ and $\eta_{1} =  1$, i.e.,

\begin{eqnarray*}
\boldsymbol{\epsilon} &=& {\bf 0} \ \epsilon_2 \epsilon_3 \ldots
  \epsilon_{n-1}, \\
\boldsymbol{\eta} &=& {\bf  1} \ \epsilon_2 \epsilon_3 \ldots 
\epsilon_{n-1}.
\end{eqnarray*}

Set $\Psi_n (\boldsymbol{\epsilon}) := a_1 a_2 \ldots a_n$ and 
suppose that  $\Psi_n (\boldsymbol{\eta}) = b_1 b_2 \ldots b_n$. 
If $\epsilon_{i-1}= 0$ with $i \ge 3$, then we see $a_i = b_i$ by 
Definition {\rmfamily \ref{def-bin-to-132-312}}. 
If $\epsilon_{i-1} = 1$ with $i \ge 3$, then $a_i$ (resp. $b_i$) 
is $n$ minus the cardinality of $j$ with $j \ge i$ and 
$\epsilon_j = 1$ (resp. $\eta_j = 1$) and  hence we see $a_i = b_i$. 
Therefore we see $a_p = b_p$ for $p \ge 3$.
Let $k$ be the cardinality of $0$ in $\boldsymbol{\epsilon}$, 
then $a_1 = (k+1), a_2 = k, b_1 = k$ and $b_2 = (k+1)$. 
 Hence 
$\Psi_n (\boldsymbol{\epsilon})$ and 
$\Psi_n (\boldsymbol{\eta})$ differ by an adjacent transposition.

\end{proof}

From Theorem {\rmfamily \ref{augment-hamilton-cycle}} and the definition of $\Psi_n$, we obtain the following result.

\begin{theorem}\label{213-312-gray-code}

For $n \ge 3$, there exists a Gray code 

\begin{equation}
(\boldsymbol{a}_1, \boldsymbol{a}_2, \cdots , \boldsymbol{a}_{2^{n-1}})
\end{equation}

for $S_n(213, 312)$ which satisfies: 

{\rm (P1)}
$\boldsymbol{a}_1 = n (n-1) \ldots 321$, the reversal of the identity permutation, 
 and 
$\boldsymbol{a}_{2^{n-1}} = (n-1) n (n-2) (n-3)  \ldots  321$, the first (resp. second) entry is 
$(n-1)$ (resp. $ n$) and the $i$-th entry is $(n-i+1)$ for $i \ge 3$,;

{\rm (P2)}
two consecutive elements differ by at most $2$ adjacent transpositions;

{\rm (P3)} 
if $ \boldsymbol{a}_i$ and $ \boldsymbol{a}_{i+1}$ differ by $2$ adjacent 
transpositions, then  
$ \boldsymbol{a}_{i-1}$ and $ \boldsymbol{a}_i $ differ by an adjacent transposition and also  
$ \boldsymbol{a}_{i+1}$ and $ \boldsymbol{a}_{i+2} $ do for $ 1 \le i \le 2^{n-1}-1$.

\end{theorem}

\begin{remark}
For $p_1 p_2 \ldots p_n \in S_n$, its reversal is defined by $p_n p_{n-1} \ldots p_1 \in S_n$ 
\cite{simion-schmidt}.
By applying the reversal for every element in the Gray code in 
Theorem  {\rmfamily \ref{213-312-gray-code}}, we obtain a Gray code which starts from 
$123 \cdots (n-1) n$ and ends at $123 \cdots (n-2) n (n-1)$ which satisfies  
{\rm P(2)} and {\rm (P3)} in Theorem {\rmfamily \ref{213-312-gray-code}}.

\end{remark}

\begin{corollary}
For $n \ge 3$, there is a Gray code for $S_n(213, 312)$ such that any successive 
permutations differ by at most two adjacent transpositions. 

\end{corollary}

\begin{remark}

On our Gray code in Theorem {\rmfamily \ref{213-312-gray-code}}, 
the first permutation and the second permutation differ by two adjacent transpositions. 
If one obtain a Hamiltonian cycle for $G(n-1)$, then one gets a Gray code for $S_n (213,312)$ such that 
two successive permutations differ by only one  adjacent transposition. 
\end{remark}

\begin{example}

From Example {\rmfamily \ref{01-listing-n=5-example}} 
and $\Psi_6$,  we obtain a following 
listing $(\boldsymbol{a}_1, \boldsymbol{a}_2, \cdots , \boldsymbol{a}_{32})$ 
for $S_6 (132,312)$.

\begin{center}
\begin{tabular}{cccccc}
$\boldsymbol{a}_1 = 654321$, &
$\boldsymbol{a}_2 = 546321$, &
$\boldsymbol{a}_3 = 456321$, &
$\boldsymbol{a}_4 = 435621$, &
$\boldsymbol{a}_5 = 345621$, &
$\boldsymbol{a}_6 = 453621$, \\
$\boldsymbol{a}_7 = 453261$, &
$\boldsymbol{a}_{8} = 453216$, &
$\boldsymbol{a}_{9} = 543216$, &
$\boldsymbol{a}_{10} = 435216$, &
$\boldsymbol{a}_{11} = 345216$, &
$\boldsymbol{a}_{12} = 324516$, \\
$\boldsymbol{a}_{13} = 234516$, &
$\boldsymbol{a}_{14} = 342516$, &
$\boldsymbol{a}_{15} = 342156$, &
$\boldsymbol{a}_{16} = 321456$, & 
$\boldsymbol{a}_{17} = 231456$, & 
$\boldsymbol{a}_{18} = 123456$, \\
$\boldsymbol{a}_{19} = 213456$, &
$\boldsymbol{a}_{20} = 234156$, &
$\boldsymbol{a}_{21} = 324156$, &
$\boldsymbol{a}_{22} = 432156$, &
$\boldsymbol{a}_{23} = 432516$, &
$\boldsymbol{a}_{24} = 432561$, \\ 
$\boldsymbol{a}_{25} = 342561$, & 
$\boldsymbol{a}_{26} = 234561$, &
$\boldsymbol{a}_{27} = 324561$, &
$\boldsymbol{a}_{28} = 345261$, &
$\boldsymbol{a}_{29} = 435261$, &
$\boldsymbol{a}_{30} = 543261$, \\
$\boldsymbol{a}_{31} = 543621$, &
$\boldsymbol{a}_{32} = 564321$. &
\end{tabular}

\end{center}

\end{example}

\section{A Gray code for  $\mathcal{A}_n (001)$ }

In this section, we construct a Gray code of strong distance $2$ for $\mathcal{A}_n(001)$ for $n \in \mathbb{N}$.
An ascent sequence avoids $001$ pattern if and only if it starts with a strictly increasing ascent sequence, we denote it by $012 \ldots k$, followed by an arbitrary weakly decreasing sequence whose letters are smaller than or equal to $k$ 
\cite{duncan-steingrimsson}.
First, we define a map $\Phi_{\mathbb{F}_2^{n-1} \rightarrow \mathcal{A}_n(001)}$ from  
$\mathbb{F}^{n-1}_2$ to $\mathcal{A}_n(001)$.

For a binary word $\boldsymbol{\alpha} = \alpha_1 \alpha_2 \ldots \alpha_{n-1}$,  we say that 
the $x$-th entry of $0$   is  
in the $i$-th position 
when $\alpha_i = 0$  and
the cardinality of $j$ with $j \le i$ and $\alpha_j = 0$ is 
$x$.
 For example, 
The $6$-th  entry of $0$  in $00010101{\bf 0}11001101$ 
is in the $9$-th position, we denote it in bold style.

For $\boldsymbol{\epsilon} \in \mathbb{F}_2^{n-1}$, we can write 

\begin{displaymath}
\boldsymbol{\epsilon} = 
\underbrace{11 \ldots 1}_{p \ \rm{times} \ \ \ \ \ }  0
\underbrace{11 \ldots 1}_{(k_1 - k_2)  \  \rm{times}} 0  
\underbrace{11 \ldots 1}_{(k_2 - k_3) \ \rm{times}} 0
\cdots 
0 
\underbrace{11 \ldots 1}_{(k_{q-1} - k_q ) \ \rm{times}}  0 
\underbrace{11 \ldots 1}_{ \ \ \ \ \ k_q \  \rm{times}},
\end{displaymath}

for some $n-1 \ge k_1 \ge  k_2 \ge \ldots \ge k_q \ge 0 $
 and $ p, q \ge 0$ such that 
$ k_1 + p + q = n-1$, where 
 $k_i$ is the number of $1$ on the right hand side of 
the $i$-th entry of $0$. We have  
$k_i = k_{i+1}$ when the $i$-th entry of $0$   is adjacent to 
the $(i+1)$-th entry of $0$. \\

We set $\phi^{(n)}_1 (\boldsymbol{\epsilon}) := 012 \ldots (k_1 + p)$ and 
$\phi^{(n)}_2 (\boldsymbol{\epsilon}) := k_1 k_2 \ldots  k_q$. 
Remark that the cardinality of $1$ is $(k_1 + p)$ and that of $0$ is $q$. 
Define $\Phi_{\mathbb{F}_2^{n-1} \rightarrow \mathcal{A}_n(001)} (\boldsymbol{\epsilon})$ to be the concatenation of 
 $\phi^{(n)}_1 (\boldsymbol{\epsilon})$ and $\phi^{(n)}_2 (\boldsymbol{\epsilon})$, that is 

\begin{displaymath}
\Phi_{\mathbb{F}_2^{n-1} \rightarrow \mathcal{A}_n(001)}
 (\boldsymbol{\epsilon}) :=012 \ldots (k_1 + p) k_1 k_{2} \ldots  k_q
\end{displaymath}

and it is a $001$ avoiding ascent sequence. 

For example, for
 $\boldsymbol{\epsilon} = 010111010011 \in \mathbb{F}_2^{12}$, 
we see $k_1 = 7,  k_2 = 6, k_3 =3, k_4 = k_5 = 2, p=0$ and  $q=5$ and 
we have $\phi^{(13)}_1 (\boldsymbol{\epsilon}) =01234567$ and 
$\phi^{(13)}_2 (\boldsymbol{\epsilon}) = 76322$. Hence we obtain   
$\Phi_{\mathbb{F}_2^{12} \rightarrow \mathcal{A}_{13} (001)} ( 010111010011) = 0123456776322$. 
We state the above observation as  Lemma. 

\begin{lemma}\label{001-avoider-key-fact}
For a binary word $\boldsymbol{\epsilon} = \epsilon_1 \epsilon_2 \ldots \epsilon_{n-1}$, 
suppose that $\phi^{(n)}_2 (\boldsymbol{\epsilon}) = k_1 k_{2} \ldots  k_q$.

{\rm (1)} \ The number of $0$ in $\boldsymbol{\epsilon}$ is $q$.

{\rm (2)} \  
For $ 1 \le i \le q$, 
the cardinality of $1$ in $\boldsymbol{\epsilon}$ 
 on the right hand side of the $i$-th entry of 
$0$ equals $k_i$.   

{\rm (3) } $\phi^{(n)}_1 (\boldsymbol{\epsilon}) = 012 \ldots (n-1-q)$, where $(n-1-q)$ is the number of $1$ in $\boldsymbol{\epsilon}$.
\end{lemma}

Next, we define a map 
$\Psi_{\mathcal{A}_n(001) \rightarrow \mathbb{F}_2^{n-1} }$ from $\mathcal{A}_n(001)$ to 
$\mathbb{F}^{n-1}_2$.
A $001$ avoiding  ascent sequence of length $n$ can be written as 
 $012 \ldots r x_{1} x_{2} \ldots x_s$, where 
$012 \ldots r$ is a strictly increasing ascent sequence and 
$x_{1} x_{2} \ldots x_s$ is a weakly decreasing sequence whose letters are 
smaller than or equal to $r$ with $r+s = n-1$.

Let 
$\Psi_{\mathcal{A}_n(001) \rightarrow \mathbb{F}_2^{n-1} }
 (012 \ldots r x_{1} x_{2} \ldots x_s)$ be the binary words of length 
$(n-1)$ such that:

(1) the cardinality of $0$ (resp. $1$) is $s$ (resp. $r$); 

(2) the cardinality of $1$ on the right hand side of the $i$-th entry of 
$0$ equals $x_i$ for $ 1 \le i \le s$, i.e.,

\begin{displaymath}
\underbrace{11 \ldots 1}_{ (r - x_1 ) \ \rm{times}} 0
\underbrace{11 \ldots 1}_{ (x_1 - x_2) \  \rm{times}} 0 
\underbrace{11 \ldots 1}_{ (x_2 - x_3) \  \rm{times}} 0 
\ldots 
0
\underbrace{11 \ldots 1}_ {(x_{s-1} - x_{s}) \ \rm{times}} 0 
\underbrace{11 \ldots 1}_{ \ \ \ \ \ x_s \ \rm{times}}.
\end{displaymath}

For  $0123422100 \in \mathcal{A}_{10}(001)$, the strictly increasing  sequence part is $01234$ and the weakly decreasing part is $22100$.  From this, we see $r=4, s=5$, $x_1 = x_2 =2, x_3=1$ and 
$x_4 = x_5 = 0$. We obtain

\begin{displaymath}
\underbrace{11}_{ r - x_1 = 2  \  \rm{times}} 0
\underbrace{\phi}_{ x_1 - x_2 = 0 \  \rm{times}} 0 
\underbrace{1}_{ x_2 - x_3 = 1 \  \rm{times}} 0 
\underbrace{1}_ {x_3 - x_4 = 1 \ \rm{times}} 0 
\underbrace{\phi}_{ x_4 - x_5 = 0 \ \rm{times}}0
\underbrace{\phi}_{x_5 = 0 \ \rm{times}}
\end{displaymath}

and $\Phi_{\mathcal{A}_{10} (001) \rightarrow \mathbb{F}_2^{9} }  (0123422100) = 110010100$. 
By the construction, the following Lemma is straightforward to check.

\begin{lemma}
$\Phi_{\mathbb{F}_2^{n-1} \rightarrow \mathcal{A}_n(001)} \circ  
\Psi_{\mathcal{A}_n(001) \rightarrow  \mathbb {F}_2}$
 (resp. $\Psi_{ \mathcal{A}_n(001) \rightarrow \mathbb{F}_2^{n-1}} \circ 
\Phi_{\mathbb{F}_2^{n-1} \rightarrow \mathcal{A}_n(001)}$) is an identity map on 
$\mathcal{A}_n(001)$ (resp.  $\mathbb{F}_2^{n-1}$). 
Hence $\Psi_{\mathbb{F}_2^{n-1} \rightarrow \mathcal{A}_n(001)}$ and 
$\Phi_{\mathbb{F}_2^{n-1} \rightarrow \mathcal{A}_n(001)}$ are bijections.

\end{lemma}

\begin{proposition}\label{main-001-avoider}
If $\boldsymbol{\epsilon} = \epsilon_1 \epsilon_2 \ldots \epsilon_{n-1}$ and 
$\boldsymbol{\eta} = \eta_1 \eta_2 \ldots \eta_{n-1}$ are adjacent in $G(n-1)$, then the strong distance of 
$\Psi_{\mathbb{F}_2^{n-1} \rightarrow \mathcal{A}_n(001)} (\boldsymbol{\epsilon})$ and 
$\Psi_{\mathbb{F}_2^{n-1} \rightarrow \mathcal{A}_n(001)} (\boldsymbol{\eta})$ is $1$.

\end{proposition}

\begin{proof}

By the assumption, $\boldsymbol{\eta}$ 
can be obtained from $\boldsymbol{\epsilon}$ by interchanging adjacent $0$ and $1$ 
in $\boldsymbol{\epsilon}$ or by changing the first position of 
$\boldsymbol{\epsilon}$. \\

\underline{Case 1.} We discuss the case when $\boldsymbol{\eta}$ is obtained by interchanging adjacent $0$ and $1$ in $\boldsymbol{\epsilon}$.
Without loss of generality, we can assume that 
$\eta_x= \epsilon_x$ for $x\neq i, i+1$, 
$\epsilon_i = \eta_{i+1} = 1$ and 
$\epsilon_{i+1} = \eta_i = 0$.

Suppose that  the $x$-th entry of $0$  in 
$\boldsymbol{\epsilon}$ is  in the $(i+1)$-th position. 
The the $x$-th entry of $0$ in $\boldsymbol{\eta}$ is in the  $i$-th position, i.e.,

\begin{eqnarray*}
\boldsymbol{\epsilon} &=& \epsilon_1 \epsilon_2 \ldots
\epsilon_{i-1} \ \ \ \ \ \ \ \ \ 
 {\bf 1} \underbrace{0}_{x- \rm{th \ entry \ of } \  0} \ 
  \epsilon_{i+2} \ldots \epsilon_{n-1}, \\
\boldsymbol{\eta} &=& \epsilon_1 \epsilon_2 \ldots 
\epsilon_{i-1} \  \underbrace{0}_{x- \rm{th \ entry \ of } \  0} 
 {\bf 1} \ \ \ \ \ \ \ \ \  \epsilon_{i+2} \ldots \epsilon_{n-1}.
\end{eqnarray*}

From (3) of Lemma {\rmfamily \ref{001-avoider-key-fact}},  we have 
$\phi^{(n)}_1 (\boldsymbol{\epsilon}) = \phi^{(n)}_1 (\boldsymbol{\eta})$.
Set $\phi^{(n)}_2 (\boldsymbol{\epsilon}) = \theta_1 \theta_{2} \ldots \theta_r$ and 
 $\phi^{(n)}_2 (\boldsymbol{\eta}) = \tau_1 \tau_2 \ldots \tau_s$. 
From {\rm (1)} of Lemma {\rmfamily \ref{001-avoider-key-fact}}, we see $r = s$.
We have $\theta_u = \tau_u$ for $1 \le u \neq x \le r$ 
and we get $\tau_x = \theta_x +  1$ from {\rm (2)} of Lemma {\rmfamily \ref{001-avoider-key-fact}}.  Hence the strong distance of  
$\Phi_{\mathbb{F}_2^{n-1} \rightarrow \mathcal{A}_n(001)} (\boldsymbol{\epsilon})$ and 
$\Phi_{\mathbb{F}_2^{n-1} \rightarrow \mathcal{A}_n(001)} (\boldsymbol{\eta})$ is $1$. \\

\underline{Case 2.} We discuss the case when $\boldsymbol{\eta}$ is obtained 
by changing the first position of  $\boldsymbol{\epsilon}$.
Without loss of generality, we can assume that 
$\eta_i= \epsilon_i$ for $ 2 \le i \le n-1$, 
$\epsilon_{1} = 1$ and $\eta_{1} =  0$, i.e.,

\begin{eqnarray*}
\boldsymbol{\epsilon} &=&
{\bf 1}  \ \epsilon_2 \epsilon_3 \ldots
  \epsilon_{n-1},  \\
\boldsymbol{\eta} &=& {\bf 0} \ \epsilon_2 \epsilon_3 \ldots 
\epsilon_{n-1}.
\end{eqnarray*}

Suppose that  the number of $1$ (resp. $0$) in 
$\boldsymbol{\epsilon}$ is $t$ (resp. $n-1-t$). 
 Then  the number of $1$ (resp. $0$) in 
$\boldsymbol{\eta}$ is $t-1$ (resp. $n-t$).

From \rm{(3)} of Lemma 
{\rmfamily \ref{001-avoider-key-fact}}, 
we see $\phi^{(n)}_1 (\boldsymbol{\epsilon}) = 012 \ldots t$ and 
$\phi^{(n)}_1 (\boldsymbol{\eta}) = 012 \ldots (t-1)$. 
Also $\phi^{(n)}_2 (\boldsymbol{\epsilon})$ has $n-1-t$ letters and 
set $\phi^{(n)}_2 (\boldsymbol{\epsilon}) = b_1 b_2 \ldots  b_{n-1-t}$. 
Then 
$\phi^{(n)}_2 (\boldsymbol{\eta}) = (t-1) b_{1} b_2 \ldots  b_{n-1-t}$, because the first entry of $0$ 
in  $\boldsymbol{\eta}$  is at the first position of $\boldsymbol{\eta}$ and 
$\eta_i = \epsilon_i$ for  $2 \le i \le n-1$.
Hence we have

\begin{eqnarray*}
\Phi_{\mathbb{F}_2^{n-1} \rightarrow \mathcal{A}_n(001)}(\boldsymbol{\epsilon}) &=& 012 \ldots (t-1) \ \ \ \ 
t \ \ \  b_1 b_2 \ldots  b_{n-1-t}, \\
\Phi_{\mathbb{F}_2^{n-1} \rightarrow \mathcal{A}_n(001)} (\boldsymbol{\eta}) &=&     012 \ldots (t-1) 
 (t-1) b_{1} b_2 \ldots b_{n-1-t}
\end{eqnarray*}

and the strong distance of  $\Phi_{\mathbb{F}_2^{n-1} \rightarrow \mathcal{A}_n(001)} (\boldsymbol{\epsilon})$ and 
$\Phi_{\mathbb{F}_2^{n-1} \rightarrow \mathcal{A}_n(001)} (\boldsymbol{\eta})$ is $1$. 

\end{proof}

From Theorem {\rmfamily \ref{augment-hamilton-cycle}}, 
 Remark {\rmfamily \ref{binary-dist-is-two}} and 
Proposition {\rmfamily \ref{main-001-avoider}}, 
we obtain the following result.

\begin{theorem}\label{A(001)-gray-code-main}

For $n \ge 3$, there exists a Gray code $( \boldsymbol{a}_1, \boldsymbol{a}_2, \ldots 
, \boldsymbol{a}_{2^{n-1}} )$ of strong distance $2$
 for $\mathcal{A}_n (001)$ which satisfies: 

\begin{enumerate}

\item  $\boldsymbol{a}_1 = 00 \ldots 0$, all entries are $0$, and 
$\boldsymbol{a}_{2^{n-1}} = 0100 \ldots 00$, the second entry is $1$ and 
the rest of the entries are all $0$; 

\item $d_{\mathrm{str}} (\boldsymbol{a}_1,  \boldsymbol{a}_2) = 2$; 

\item  if  $d_{\mathrm{str}} (\boldsymbol{\alpha}_i, 
\boldsymbol{\alpha}_{i+1}) = 2$, then 
 $d_{\mathrm{str}} (\boldsymbol{\alpha}_{i-1} , \boldsymbol{\alpha}_{i}) 
= d_{\mathrm{str}} (\boldsymbol{\alpha}_{i+1}, \boldsymbol{\alpha}_{i+2})
=1$ for $1 \le i \le 2^{n-1}-1$.

\end{enumerate}

\end{theorem}

\begin{corollary}
For $n \ge 3$, 
$\mathcal{A}_n(001)$ has a Gray code of strong distance  $2$
 which starts from 
$00 \ldots 0$ and ends at $0100 \ldots 0$. 

\end{corollary}

\begin{remark}
For $n \ge 3$, 
if one obtain a Hamiltonian cycle for $G(n-1)$, then one gets a Gray code of strong distance $1$ for 
$\mathcal{A}_n (001)$.
\end{remark}

\begin{example}
By applying $\Phi_{\mathbb{F}_2^{5} \rightarrow \mathcal{A}_6(001)}$ to 
  the listing in Example {\rmfamily \ref{01-listing-n=5-example}}, 
we have a Gray code 
 $(\boldsymbol{a}_1, \boldsymbol{a}_2, \cdots , \boldsymbol{a}_{32})$
for $\mathcal{A}_6(001)$, where

\begin{center}
\begin{tabular}{cccccc}
$\boldsymbol{a}_1 = 000000$, &
$\boldsymbol{a}_2 = 011000$, &
$\boldsymbol{a}_3 = 012000$, &
$\boldsymbol{a}_4 = 012200$, &
$\boldsymbol{a}_5 = 012300$, &
$\boldsymbol{a}_6 = 012100$, \\
$\boldsymbol{a}_7 = 012110$, &
$\boldsymbol{a}_{8} = 012111$, &
$\boldsymbol{a}_{9} = 011111$, &
$\boldsymbol{a}_{10} =012211$, &
$\boldsymbol{a}_{11} =012311$, &
$\boldsymbol{a}_{12} =012331$, \\
$\boldsymbol{a}_{13} =012341$, &
$\boldsymbol{a}_{14} =012321$, &
$\boldsymbol{a}_{15} =012322$, &
$\boldsymbol{a}_{16} =012333$, & 
$\boldsymbol{a}_{17} =012343$, & 
$\boldsymbol{a}_{18} =012345$, \\
$\boldsymbol{a}_{19} = 012344$, &
$\boldsymbol{a}_{20} = 012342$, &
$\boldsymbol{a}_{21} = 012332$, &
$\boldsymbol{a}_{22} = 012222$, &
$\boldsymbol{a}_{23} = 012221 $, &
$\boldsymbol{a}_{24} = 012220$, \\ 
$\boldsymbol{a}_{25} = 012320$, & 
$\boldsymbol{a}_{26} = 012340$, &
$\boldsymbol{a}_{27} = 012330$, &
$\boldsymbol{a}_{28} = 012310$, &
$\boldsymbol{a}_{29} = 012210$, &
$\boldsymbol{a}_{30} = 011110$, \\
$\boldsymbol{a}_{31} =011100 $, &
$\boldsymbol{a}_{32} = 010000$. &
\end{tabular}
\end{center}

\end{example}

\section{A Gray code for $\mathcal{A}_n (010)$ }

In this section, we construct a Gray code of strong distance $2$ for $\mathcal{A}_n(010)$ for $n \in \mathbb{N}$.
An ascent sequence avoids $010$ pattern if and only if 
it is a weakly increasing ascent sequence \cite{duncan-steingrimsson}.

For a binary word $\boldsymbol{\epsilon} = \epsilon_1 \epsilon_2 \ldots \epsilon_{n-1}$, let 
 $a_1 a_2 \ldots a_n$ be the sequence  such that $a_1 =0$ and  
 $a_i$ is the number of $1$ in $\epsilon_{n-i+1} \epsilon_{n-i+2} \ldots 
\epsilon_{n-1}$ for $2 \le i \le n$. The sequence is weakly increasing 
and hence it is a $010$ avoiding ascent sequence. 
Define 
$\Phi_{\mathbb{F}_2^{n-1} \rightarrow \mathcal{A}_n(010)}  (\boldsymbol{\epsilon}) := a_1 a_2 \ldots a_n$.   
  Obviously, it is a bijection from $\mathbb{F}_2^{n-1}$ to 
$\mathcal{A}_n (010)$. 
For example, 
$\Phi_{\mathbb{F}_2^{n-1} \rightarrow \mathcal{A}_n(010)} 
 (100110100) = 0001123334$.

\begin{proposition}\label{010-avoider-main}
If $\boldsymbol{\epsilon} = \epsilon_1 \epsilon_2 \ldots \epsilon_{n-1}$ and 
$\boldsymbol{\eta} = \eta_1 \eta_2 \ldots \eta_{n-1}$ are adjacent in $G(n-1)$, then the strong distance of 
$\Psi_{\mathbb{F}_2^{n-1} \rightarrow \mathcal{A}_n(010)}  (\boldsymbol{\epsilon})$ and 
$\Psi_{\mathbb{F}_2^{n-1} \rightarrow \mathcal{A}_n(010)}  (\boldsymbol{\eta})$ is $1$.
\end{proposition}

\begin{proof}

By the assumption, $\boldsymbol{\eta}$ 
can be obtained from $\boldsymbol{\epsilon}$ by interchanging adjacent $0$ and $1$ 
in $\boldsymbol{\epsilon}$ or by changing the first position of $\boldsymbol{\epsilon}$. \\

\underline{Case 1.} We consider the case where $\boldsymbol{\eta}$ is obtained by interchanging adjacent $0$ and $1$ in $\boldsymbol{\epsilon}$.
Without loss of generality, we can assume that 
$\eta_x= \epsilon_x$ for $x\neq i, i+1$, 
$\epsilon_i = \eta_{i+1} = 1$ and 
$\epsilon_{i+1} = \eta_i = 0$ for some $1 \le i \le (n-1)$, i.e.,

\begin{eqnarray*}
\boldsymbol{\epsilon} &=& \epsilon_1 \epsilon_2 \ldots
\epsilon_{i-1} {\bf 1} {\bf 0} 
  \epsilon_{i+2} \ldots \epsilon_{n-1}, \\
\boldsymbol{\eta} &=& \epsilon_1 \epsilon_2 \ldots 
\epsilon_{i-1}  {\bf 0} {\bf 1}  \epsilon_{i+2} \ldots \epsilon_{n-1}.
\end{eqnarray*}

Put $\Phi (\boldsymbol{\epsilon}) = 
a_1 a_2  \ldots  a_n$. Then  
$a_{n-i} = a_{n-i-1}$ and $a_{n-i+1} = a_{n-i}+1$. 
If $\Phi (\boldsymbol{\epsilon}) = b_1 b_2 \ldots b_n$, then by the 
construction, we see 
$b_k = a_k$ for $k \neq (n-i), (n-i+1)$. 
Also we have $b_{n-i} = b_{n-i-1} + 1$ and $b_{n-i+1} = b_{n-i}$ and 
hence we obtain $a_{n-i+1} = b_{n-i+1}$ and $a_{n-i} = b_{n-i} -1$.
Therefore the  
 strong distance of 
$\Phi_{\mathbb{F}_2^{n-1} \rightarrow \mathcal{A}_n(010)}  (\boldsymbol{\epsilon})$ and 
$\Phi_{\mathbb{F}_2^{n-1} \rightarrow \mathcal{A}_n(010)}  (\boldsymbol{\eta})$ is $1$.\\

\underline{Case 2.} We discuss the case where $\boldsymbol{\eta}$ is obtained 
by changing the first position of  $\boldsymbol{\epsilon}$.
Without loss of generality, we can assume that 
$\eta_i= \epsilon_i$ for $ 2 \le i \le (n-1)$, 
$\epsilon_{1} = 0$ and $\eta_{1} =  1$, i.e.,

\begin{eqnarray*}
\boldsymbol{\epsilon} &=& {\bf  0} \ \epsilon_2 \epsilon_3 \ldots
  \epsilon_{n-1},  \\
\boldsymbol{\eta} &=& {\bf  1} \ \epsilon_2 \epsilon_3 \ldots 
\epsilon_{n-1}.
\end{eqnarray*}

If $\Phi_{\mathbb{F}_2^{n-1} \rightarrow \mathcal{A}_n(010)} (\boldsymbol{\epsilon}) = 
	a_1 a_2  \ldots  a_{n-1} a_n$, then we see    
$\Phi_{\mathbb{F}_2^{n-1} \rightarrow \mathcal{A}_n(010)} (\boldsymbol{\eta}) = 
a_1 a_2  \ldots  a_{n-1} (a_{n}+1)$.   
Therefore  the strong distance of 
$\Phi_{\mathbb{F}_2^{n-1} \rightarrow \mathcal{A}_n(001)} (\boldsymbol{\epsilon})$ and 
$\Phi_{\mathbb{F}_2^{n-1} \rightarrow \mathcal{A}_n(001)} (\boldsymbol{\eta})$ is $1$.

\end{proof}

From Theorem {\rmfamily \ref{augment-hamilton-cycle}}, 
 Remark {\rmfamily \ref{binary-dist-is-two}} and 
Proposition {\rmfamily \ref{010-avoider-main}}, 
we obtain the following result.

\begin{theorem}\label{A(010)-gray-code-main}

For $n \ge 3$, there is a Gray code $( \boldsymbol{\alpha}_1, 
\boldsymbol{\alpha}_2, \ldots , \boldsymbol{\alpha}_{2^{n-1}} )$ of strong 
distance $2$ for 
$\mathcal{A}_n (010)$ which satisfies: 

\begin{enumerate}

\item  $\boldsymbol{\alpha}_1 = 00 \ldots 0$, all entries are $0$, 
and $\boldsymbol{\alpha}_{2^{n-1}} = 00 \ldots 01$, the last entry 
is $1$ and the remaining entries are all $0$; 

\item $d_{\mathrm{str}} (\boldsymbol{a}_1, \boldsymbol{a}_2 ) =2$;

\item  if  $d_{\mathrm{str}} (\boldsymbol{\alpha}_i, 
\boldsymbol{\alpha}_{i+1}) = 2$, then 
 $d_{\mathrm{str}} (\boldsymbol{\alpha}_{i-1} , \boldsymbol{\alpha}_{i}) 
= d_{\mathrm{str}} (\boldsymbol{\alpha}_{i+1}, \boldsymbol{\alpha}_{i+2})
=1$ for $1 \le i \le 2^{n-1}-1$. 

\end{enumerate}
\end{theorem}

\begin{corollary}
For $n \ge 3$, 
$\mathcal{A}_n(010)$ has a Gray code of strong distance $2$
 which starts from 
$00 \ldots 0$ and ends at $000 \ldots 01$. 
\end{corollary}

\begin{remark}
For $n \ge 3$, if one obtain a Hamiltonian cycle for $G(n-1)$, then one gets a Gray code 
of strong distance $1$ for $\mathcal{A}_n (010)$.

\end{remark}

\begin{example}
By applying $\Phi_{\mathbb{F}_2^{5} \rightarrow \mathcal{A}_6 (010)}$ 
to   the listing in Example {\rmfamily \ref{01-listing-n=5-example}}, 
we have a Gray code 
 $(\boldsymbol{a}_1, \boldsymbol{a}_2, \cdots , \boldsymbol{a}_{32})$
for $\mathcal{A}_6(010)$, where 

\begin{center}
\begin{tabular}{cccccc}
$\boldsymbol{a}_1 = 000000$, &
$\boldsymbol{a}_2 = 000011$, &
$\boldsymbol{a}_3 = 000012$, &
$\boldsymbol{a}_4 = 000122$, &
$\boldsymbol{a}_5 = 000123$, &
$\boldsymbol{a}_6 = 000112$, \\
$\boldsymbol{a}_7 = 001112$, &
$\boldsymbol{a}_{8} = 011112$, &
$\boldsymbol{a}_{9} = 011111$, &
$\boldsymbol{a}_{10} =011122$, &
$\boldsymbol{a}_{11} =011123$, &
$\boldsymbol{a}_{12} =011233$, \\
$\boldsymbol{a}_{13} =011234$, &
$\boldsymbol{a}_{14} =011223$, &
$\boldsymbol{a}_{15} =012223$, &
$\boldsymbol{a}_{16} =012333$, & 
$\boldsymbol{a}_{17} =012334$, & 
$\boldsymbol{a}_{18} =012345$, \\
$\boldsymbol{a}_{19} =012344$, &
$\boldsymbol{a}_{20} =012234$, &
$\boldsymbol{a}_{21} =012233$, &
$\boldsymbol{a}_{22} =012222$, &
$\boldsymbol{a}_{23} =011222 $, &
$\boldsymbol{a}_{24} =001222$, \\ 
$\boldsymbol{a}_{25} =001223$, & 
$\boldsymbol{a}_{26} =001234$, &
$\boldsymbol{a}_{27} =001233$, &
$\boldsymbol{a}_{28} =001123$, &
$\boldsymbol{a}_{29} =001122$, &
$\boldsymbol{a}_{30} =001111$, \\
$\boldsymbol{a}_{31} = 000111$, &
$\boldsymbol{a}_{32} =000001$. &
\end{tabular}
\end{center}

\end{example}

{\bf Acknowledgement}

Theorem  {\rmfamily \ref{A(001)-gray-code-main}} and {\rmfamily \ref{A(010)-gray-code-main}} are presented at 
Permutations Patterns 2017 in Reykjavik.  
The author wishes to thank Antonio Bernini and Tomoki Yamashita for their valuable comments.

\end{document}